\documentclass{amsart}
\usepackage{latexsym}
\usepackage{amssymb}
\usepackage{amsfonts}
\usepackage{amsmath}
\usepackage{graphicx}
\newtheorem{thm}{Theorem}[section]
\newtheorem{corollary}[thm]{Corollary}
\newtheorem{lemma}[thm]{Lemma}

\theoremstyle{definition}
\newtheorem{definition}[thm]{Definition}
\newtheorem{example}[thm]{Example}

\theoremstyle{remark}
\newtheorem{remark}[thm]{Remark}
\newtheorem{claim}{Claim}
\DeclareMathOperator{\qf}{qf}

\DeclareMathOperator{\Max}{Max}
\DeclareMathOperator{\Cl}{Cl}
\DeclareMathOperator{\Pic}{Pic}
\numberwithin{equation}{section}
\newcommand{\field}[1]{\mathbb{#1}}

\newcommand{\R}{\field{R}}

\begin{document}
%%%%%%%%%%%%%%%%%%%%%%%%%%%%%%%%%%%%%%%%%%%%%%%%%%%%%%%%%
%%%%%%%%%%%%%%%%%%%%%%%%%%%%%%%%%%%%%%%%%%%%%%%%%%%%%%%%%
\title[Integral domains with Boolean $t$-Class semigroup]{Integral domains with Boolean $t$-Class semigroup}

%%%%%%%%%%%%%%%%%%%%%%%%%%%%%%%%%%%%%%%%%%%%%%%%%%%%%%%%%
%%%%%%%%%%%%%%%%%%%%%%%%%%%%%%%%%%%%%%%%%%%%%%%%%%%%%%%%%
\author{S. Kabbaj}

\address{Department of Mathematics and Statistics,
King Fahd University of Petroleum \& Minerals, P. O. Box 5046,
Dhahran 31261, Saudi Arabia}

\email{kabbaj@kfupm.edu.sa}

%%%%%%%%%%%%%%%%%%%%%%%%%%%%%%%%%%%%%%%%%%%%%%%%%%%%%%%%%
%%%%%%%%%%%%%%%%%%%%%%%%%%%%%%%%%%%%%%%%%%%%%%%%%%%%%%%%%
\author{A. Mimouni}

\address{Department of Mathematics and Statistics,
King Fahd University of Petroleum \& Minerals, P. O. Box 5046,
Dhahran 31261, Saudi Arabia}

\email{amimouni@kfupm.edu.sa}

%%%%%%%%%%%%%%%%%%%%%%%%%%%%%%%%%%%%%%%%%%%%%%%%%%%%%%%%%
%%%%%%%%%%%%%%%%%%%%%%%%%%%%%%%%%%%%%%%%%%%%%%%%%%%%%%%%%
\thanks{This work was funded by King Fahd University of Petroleum \& Minerals under Project \# MS/t-Class/257.}

\date{\today}

%%%%%%%%%%%%%%%%%%%%%%%%%%%%%%%%%%%%%%%%%%%%%%%%%%%%%%%%%
%%%%%%%%%%%%%%%%%%%%%%%%%%%%%%%%%%%%%%%%%%%%%%%%%%%%%%%%%

%%%%%%%%%%%%%%%%%%%%%%%%%%%%%%%%%%%%%%%%%%%%%%%%%%%%%%%%%
%%%%%%%%%%%%%%%%%%%%%%%%%%%%%%%%%%%%%%%%%%%%%%%%%%%%%%%%%

\dedicatory{}

%%%%%%%%%%%%%%%%%%%%%%%%%%%%%%%%%%%%%%%%%%%%%%%%%%%%%%%%%
%%%%%%%%%%%%%%%%%%%%%%%%%%%%%%%%%%%%%%%%%%%%%%%%%%%%%%%%%
\begin{abstract}
The $t$-class semigroup of an integral domain is the semigroup of the isomorphy classes of the $t$-ideals with the operation induced
by $t$-multiplication. This paper investigates  integral domains with Boolean $t$-class semigroup with an emphasis on the GCD and stability conditions. The main results establish $t$-analogues for well-known results on Pr\"ufer domains and B\'ezout domains of finite character.
\end{abstract}

\maketitle

%%%%%%%%%%%%%%%%%%%%%%%%%%%%%%%%%%%%%%%%%%%%%%%%%%%%%%%%%%%%%%%%%%%%%%%%%%%%%%%%%%%%%%%%%%%%%%%%%%%%%%%%%%%%%%%%%%%%%%%%%%%%%%%%%%%%%%%%%%%%%%
%%%%%%%%%%%%%%%%%%%%%%%%%%%%%%%%%%%%%%%%%%%%%%%%%%%%%%%%%%%%%%%%%%%%%%%%%%%%%%%%%%%%%%%%%%%%%%%%%%%%%%%%%%%%%%%%%%%%%%%%%%%%%%%%%%%%%%%%%%%%%%
%%%%%%%%%%%%%%%%%%%%%%%%%%%%%%%%%%%%%%%%%%%%%%%%%%%%%%%%%%%%%%%%%%%%%%%%%%%%%%%%%%%%%%%%%%%%%%%%%%%%%%%%%%%%%%%%%%%%%%%%%%%%%%%%%%%%%%%%%%%%%%
%%%%%%%%%%%%%%%%%%%%%%%%%%%%%%%%%%%%%%%%%%%%%%%%%%%%%%%%%%%%%%%%%%%%%%%%%%%%%%%%%%%%%%%%%%%%%%%%%%%%%%%%%%%%%%%%%%%%%%%%%%%%%%%%%%%%%%%%%%%%%%
\section{Introduction}

\noindent All rings considered in this paper are integral domains (i.e., commutative with identity and without zero-divisors). The class semigroup of a domain $R$, denoted $S(R)$, is the semigroup of nonzero fractional ideals modulo its subsemigroup of nonzero principal ideals \cite{BS,ZZ}. The $t$-class semigroup of $R$, denoted $S_{t}(R)$, is the semigroup of fractional $t$-ideals modulo its subsemigroup of nonzero principal ideals, that is, the semigroup of the isomorphy classes of the $t$-ideals of $R$ with the operation induced by ideal $t$-multiplication. Notice that $S_{t}(R)$ is the $t$-analogue of $S(R)$, as the class group $\Cl(R)$ is the $t$-analogue of the Picard group $\Pic(R)$. The following set-theoretic inclusions always hold: $\Pic(R)\subseteq \Cl(R)\subseteq \mathbf{S_{t}(R)}\subseteq S(R)$.
Note that the first and third inclusions turn into equality for Pr\"ufer domains and the second does so for Krull domains. More details on these objects are provided in the next section.

Divisibility properties of a domain $R$ are often reflected in group or semigroup-theoretic properties of $\Cl(R)$ or $S(R)$. For instance, a Pr\"ufer (resp., Krull, PVMD) domain $R$ is B\'ezout (resp., UFD, GCD) if and only if $\Cl(R)=0$ \cite{BZ}. Also if $R$ is a Dedekind domain (resp., PID), then $S(R)$ is a Clifford (resp., Boolean) semigroup. Recently, we showed that $S_{t}(R)$ is a Clifford semigroup for any Krull domain $R$; and a domain $R$ is a UFD if and only if $R$ is Krull and $S_{t}(R)$ is a Boolean semigroup \cite[Proposition 2.2]{KM2}. Recall for convenience that a commutative semigroup $S$ is Clifford if every element $x$ of $S$ is (von Neumann) regular, i.e., there exists $a\in S$ such that $x^{2}a=x$. The importance of a Clifford semigroup $S$ resides in its ability to stand as a disjoint union of subgroups $G_e$, where $e$ ranges over the set of idempotent elements of $S$, and $G_e$ is the largest subgroup of $S$ with identity equal to $e$ (Cf. \cite{Ho}). The semigroup $S$ is said to be Boolean if for each $x\in S$, $x=x^{2}$.

A domain $R$ is called a GCD domain if every pair of (nonzero) elements of $R$ has a greatest common divisor; equivalently, if the $t$-closure of any nonzero finitely generated fractional ideal of $R$ is principal \cite{Ad2}. UFDs, B\'ezout domains and polynomial rings over them are GCD domains. Ideal $t$-multiplication converts the notion of B\'ezout (resp., Pr\"ufer) domain of finite character to GCD domain of finite $t$-character (resp., Krull-type domain). A domain is stable if each nonzero ideal is invertible in its endomorphism ring (see more details in Section 2). Stability plays a crucial role in the study of class and $t$-class semigroups. Indeed, a stable domain has Clifford class semigroup \cite[Proposition 2.2]{Ba4} and finite character \cite[Theorem 3.3]{O3}; and an integrally closed stable domain is Pr\"ufer \cite[Lemma F]{ES}. Of particular relevance to our study is Olberding's result that an integrally closed domain R is stable if and only if R is a strongly discrete Pr\"ufer domain of finite character \cite[Theorem 4.6]{O1}. An analogue to this result is stated for B\'ezout domains of finite character in \cite[Theorem 3.2]{KM,KM1}.

Recall that a valuation domain has Clifford class semigroup (Bazzoni-Salce \cite{BS}); and an integrally closed domain $R$ has
Clifford class semigroup if and only if $R$ is Pr\"ufer of finite character (Bazzoni \cite[Theorem 2.14]{Ba1} and
\cite[Theorem 4.5]{Ba4}). In 2007, we extended these results to PVMDs; namely, a PVMD $R$ has Clifford $t$-class
semigroup if and only if $R$ is a Krull-type domain \cite[Theorem 3.2]{KM2}; and conjectured that this result extends to
$v$-domains (definition below). Recently, Halter-Koch solved this conjecture by using the language of ideal systems
on cancellative commutative monoids. He proved that every $t$-Clifford regular $v$-domain is a Krull-type
domain \cite[Proposition 6.11 and Proposition 6.12]{HK2}. Finally, recall Zanardo-Zannier's crucial result that an integrally closed domain with Clifford class semigroup is necessarily Pr\"ufer \cite{ZZ}. In \cite{KM}, we stated a Boolean analogue for this result, that is, an integrally closed domain with Boolean class semigroup is B\'ezout. However, in \cite[Example 2.8]{KM2}, we showed that an integrally closed domain with Boolean $t$-class semigroup need not be a PVMD (a fortiori, nor a GCD). Consequently, the class of  $v$-domains offers a natural context for studying $t$-class semigroups.

Recall from \cite{AHZ} that the pseudo-integral closure of a domain $R$ is defined as $\widetilde{R}=\bigcup(I_{t}\colon I_{t})$, where
$I$ ranges over the set of finitely generated ideals of $R$; and $R$ is said to be a  $v$-domain (or pseudo-integrally closed) if $R=\widetilde{R}$ or, equivalently, if $(I_{v}:I_{v}) = R$ for each nonzero finitely generated ideal $I$ of $R$.
A $v$-domain is called in Bourbaki's language  regularly integrally closed \cite[Ch.VII, Exercise 30]{Bou}.
Notice that $\overline{R}\subseteq \widetilde{R}\subseteq R^{\star}$, where $\overline{R}$ and $R^{\star}$ are respectively
the integral closure and the complete integral closure of $R$; and a PVMD is a $v$-domain. For recent developments on $v$-domains, we refer the reader to \cite{AAFZ,FHP,FZ,HK1,HK2}.

This paper studies $v$-domains with Boolean $t$-class semigroup with an emphasis on the GCD and
stability conditions. Our aim is to establish Boolean analogues for the aforementioned results on Pr\"ufer and B\'ezout domains of finite character. The first main result (Theorem~\ref{sec:2.1}) asserts that ``\emph{a $v$-domain with Boolean $t$-class semigroup is GCD with finite $t$-character}." Then  Corollay~\ref{sec:2.2.1} provides a Boolean analogue for (the necessity part of) Bazzoni's result mentioned above. The converse does not hold in general even for valuation domains (Remark~\ref{sec:2.2.2}). The second main result (Theorem~\ref{sec:2.3}) states a correlation between the Boolean property and stability, i.e., ``\emph{a $v$-domain has Boolean $t$-class semigroup and is strongly $t$-discrete if and only if it is strongly $t$-stable}." The third main result (Theorem~\ref{sec:2.8}) examines the class of strongly $t$-discrete domains; namely, ``\emph{assume  $R$ is a $v$-domain. Then $R$ is a strongly $t$-discrete Boole $t$-regular domain if and only if $R$ is a strongly $t$-discrete GCD domain of finite $t$-character  if and only if $R$ is a strongly $t$-stable domain}." Then Corollay~\ref{sec:2.9} recovers and improves \cite[Theorem 3.2]{KM,KM1} which provides a Boolean analogue for Olberding's result \cite[Theorem 4.6]{O1} on Pr\"ufer domains. The corollary also may be viewed as an analogue for Bazzoni's result \cite[Theorem 4.5]{Ba4} in the context of strongly discrete domains. We close with a simple method to build a new family of integral domains with Boolean $t$-class semigroup stemming from the class of GCD domains.

%%%%%%%%%%%%%%%%%%%%%%%%%%%%%%%%%%%%%%%%%%%%%%%%%%%%%%%%%%%%%%%%%%%%%%%%%%%%%%%%%%%%%%%%%%%%%%%%%%%%%%%%%%%%%%%%%%%%%%%%%%%%%%%%%%%%%%%%%%%%%
%%%%%%%%%%%%%%%%%%%%%%%%%%%%%%%%%%%%%%%%%%%%%%%%%%%%%%%%%%%%%%%%%%%%%%%%%%%%%%%%%%%%%%%%%%%%%%%%%%%%%%%%%%%%%%%%%%%%%%%%%%%%%%%%%%%%%%%%%%%%%%
%%%%%%%%%%%%%%%%%%%%%%%%%%%%%%%%%%%%%%%%%%%%%%%%%%%%%%%%%%%%%%%%%%%%%%%%%%%%%%%%%%%%%%%%%%%%%%%%%%%%%%%%%%%%%%%%%%%%%%%%%%%%%%%%%%%%%%%%%%%%%%
%%%%%%%%%%%%%%%%%%%%%%%%%%%%%%%%%%%%%%%%%%%%%%%%%%%%%%%%%%%%%%%%%%%%%%%%%%%%%%%%%%%%%%%%%%%%%%%%%%%%%%%%%%%%%%%%%%%%%%%%%%%%%%%%%%%%%%%%%%%%%%
\section{Main results}\label{sec:2}

Let $R$ be a domain with quotient field $K$ and $I$ a nonzero fractional ideal of $R$. Let $$I^{-1}:=(R:I)=\{x\in K\mid xI\subseteq R\}.$$ The $v$- and $t$-operations of $I$ are defined, respectively, by $$I_v:=(I^{-1})^{-1} \hbox{ and } I_t:=\bigcup J_v$$ where $J$ ranges over the set of finitely generated subideals of $I$. The ideal $I$ is called a $v$-ideal if $I_v=I$ and a $t$-ideal if $I_t=I$. Under the ideal $t$-multiplication $(I,J)\mapsto (IJ)_t$, the set $F_{t}(R)$ of fractional $t$~-ideals of $R$ is a semigroup with unit $R$. The set $Inv_{t}(R)$ of $t$-invertible fractional $t$-ideals of $R$ is a group with unit $R$ (Cf. \cite{Gi}). Let $F(R)$, $Inv(R)$, and $P(R)$ denote the sets of nonzero, invertible, and nonzero principal fractional ideals of $R$, respectively. Under this notation, the Picard group \cite{A,BM,G}, class group \cite{B,BZ}, $t$-class semigroup \cite{KM2}, and class semigroup \cite{BS,KM,KM1,ZZ} of $R$ are defined as follows:
\[\Pic(R):=\frac{Inv(R)}{P(R)}\ ;\
\Cl(R):=\frac{Inv_{t}(R)}{P(R)}\ ;\
\mathbf{S_{t}(R):=\frac{F_{t}(R)}{P(R)}}\ ;\
S(R):=\frac{F(R)}{P(R)}.\]

%%%%%%%%%%%%%%%%%%%%%%%%%%%%%%%%%%%%%%%%%%%%%%%%%%%%%%%%%%%%%%%%%%%%%%
%%%%%%%%%%%%%%%%%%%%%%%%%%%%%%%%%%%%%%%%%%%%%%%%%%%%%%%%%%%%%%%%%%%%%%
\begin{definition}[\cite{Ba4,BaKa,KM,KM2}]
Let $R$ be a domain.
\begin{enumerate}
\item $R$ is Clifford (resp., Boole) regular if $S(R)$ is a Clifford (resp., Boolean) semigroup.
\item $R$ is Clifford (resp., Boole) $t$-regular if $S_{t}(R)$ is a Clifford (resp., Boolean) semigroup.
\end{enumerate}
\end{definition}

A first correlation between regularity and stability conditions can be sought through Lipman stability. Indeed, $R$ is called an L-stable domain if $\bigcup_{n\geq 1} (I^{n}:I^{n})=(I:I)$ for every nonzero ideal $I$ of $R$ \cite{AHP}. Lipman introduced the notion of stability in the specific setting of one-dimensional commutative semi-local Noetherian rings in order to give a characterization of Arf rings; in this context, L-stability coincides with Boole regularity \cite{Lip}. A domain $R$ is stable (resp., strongly stable) if each nonzero ideal of $R$ is invertible (resp., principal) in its endomorphism ring \cite{AHP,KM}. Sally and Vasconcelos \cite{SV} used stability to settle Bass' conjecture on one-dimensional Noetherian rings with finite integral closure. Recent developments on this concept, due to Olberding \cite{O1,O2,O3}, prepared the ground to address the correlation between stability and the theory of class semigroups. By analogy, we define $t$-stability as a natural condition that best suits $t$-regularity:

%%%%%%%%%%%%%%%%%%%%%%%%%%%%%%%%%%%%%%%%%%%%%%%%%%%%%%%%%%%%%%%%%%%%%%
%%%%%%%%%%%%%%%%%%%%%%%%%%%%%%%%%%%%%%%%%%%%%%%%%%%%%%%%%%%%%%%%%%%%%%
\begin{definition}[\cite{KM4}]
Let $R$ be a domain.
\begin{enumerate}
\item  $R$ is $t$-stable if each $t$-ideal of $R$ is invertible in its endomorphism ring.
\item  $R$ is strongly $t$-stable if each $t$-ideal of $R$ is  principal in its endomorphism ring.
\end{enumerate}
\end{definition}

The main purpose of this work is to correlate Boole $t$-regularity with the GCD property or strong $t$-stability in the class of $v$-domains, extending known results on B\'ezout domains and stability. The first main result of this paper (Theorem~\ref{sec:2.1}) establishes a correlation between Boole $t$-regularity and GCD-domains of finite $t$-character. Recall that a domain $R$ is of  finite $t$-character if each proper $t$-ideal of $R$ is contained in only finitely many $t$-maximal ideals of $R$.

%%%%%%%%%%%%%%%%%%%%%%%%%%%%%%%%%%%%%%%%%%%%%%%%%%%%%%%%%%%%%%%%%%%%%%
%%%%%%%%%%%%%%%%%%%%%%%%%%%%%%%%%%%%%%%%%%%%%%%%%%%%%%%%%%%%%%%%%%%%%%
\begin{thm}\label{sec:2.1}
Let $R$ be a $v$-domain. If $R$ is Boole $t$-regular, then $R$ is a GCD domain of finite $t$-character.
\end{thm}

\begin{proof}
Let $I$ be a finitely generated ideal of $R$. Since $R$ is a $v$-domain, then $(I_{t}:I_{t})=R$. Since $R$ is Boole
$t$-regular, there exists $0\not =c\in \qf(R)$ such that
$(I^{2})_{t}=cI_{t}$. Hence
$(I_{t}:(I^{2})_{t})=(I_{t}:cI_{t})=c^{-1}(I_{t}:I_{t})=c^{-1}R$. On
the other hand,
$(I_{t}:(I^{2})_{t})=(I_{t}:(I_{t})^{2})=((I_{t}:I_{t}):I_{t})=(R:I_{t})=I^{-1}$.
Hence $I^{-1}=c^{-1}R$. Therefore $I_{v}=cR$, and hence $R$ is a GCD domain. Now, $R$ is a PVMD and Clifford $t$-regular, so $R$ has finite $t$-character by \cite[Theorem 3.2]{KM2}.
\end{proof}

Next, as an application of Theorem~\ref{sec:2.1}, we provide a Boolean analogue for Bazzoni's result \cite[Theorem 4.5]{Ba4} on Clifford regularity. Here we mean the necessity part of this result, since the sufficiency part \cite[Theorem 2.14]{Ba1} does not hold in general for Boole regularity, as shown below.

%%%%%%%%%%%%%%%%%%%%%%%%%%%%%%%%%%%%%%%%%%%%%%%%%%%%%%%%%%%%%%%%%%%%%%
%%%%%%%%%%%%%%%%%%%%%%%%%%%%%%%%%%%%%%%%%%%%%%%%%%%%%%%%%%%%%%%%%%%%%%
\begin{corollary}\label{sec:2.2.1}
Let $R$ be an integrally closed domain. If $R$ is Boole regular, then $R$ is a B\'ezout domain of finite character.
\end{corollary}

\begin{proof}
Recall first that an integrally closed Boole regular domain is B\'ezout \cite[Proposition 2.3]{KM}. Now, in a B\'ezout domain, the $t$-operation coincides with the trivial operation. So Theorem~\ref{sec:2.1} leads to the conclusion.
\end{proof}

%%%%%%%%%%%%%%%%%%%%%%%%%%%%%%%%%%%%%%%%%%%%%%%%%%%%%%%%%%%%%%%%%%%%%%
%%%%%%%%%%%%%%%%%%%%%%%%%%%%%%%%%%%%%%%%%%%%%%%%%%%%%%%%%%%%%%%%%%%%%%
\begin{remark}\label{sec:2.2.2}
The converses of Corollary~\ref{sec:2.2.1} and, a fortiori, Theorem~\ref{sec:2.1} are not true in general even in the context of valuation domains. To see this, recall that any rank-one non-discrete valuation domain $V$ with value group $\Gamma(V)\ncong\R$ has necessarily a non-trivial constituent group. So $S_{t}(V)$ is Clifford but not Boolean \cite[Example 3, p. 142]{BS}.
\end{remark}

A domain $R$ is strongly $t$-discrete if it has no $t$-idempotent $t$-prime ideals, i.e., for every  $t$-prime ideal $P$ of $R$,
$(P^{2})_{t}\subsetneqq P$ \cite{E,KM2}. One can easily check that a $t$-stable domain is $t$-strongly discrete; and a strongly $t$-stable domain is Boole $t$-regular. The second main result of this paper (Theorem~\ref{sec:2.3}) shows that the $t$-strongly discrete property measures how far a Boole $t$-regular domain is from being strongly $t$-stable.

%%%%%%%%%%%%%%%%%%%%%%%%%%%%%%%%%%%%%%%%%%%%%%%%%%%%%%%%%%%%%%%%%%
%%%%%%%%%%%%%%%%%%%%%%%%%%%%%%%%%%%%%%%%%%%%%%%%%%%%%%%%%%%%%%%%%%
\begin{thm}\label{sec:2.3}
Let $R$ be a $v$-domain. Then $R$ is Boole $t$-regular and strongly $t$-discrete if and only if $R$ is strongly $t$-stable.
\end{thm}

The proof of this theorem requires the following preparatory lemmas. Throughout, $v_{1}$ and $t_{1}$ will denote the $v$- and $t$-operations with respect to an overring $T$ of $R$. Also recall that $T$ is called a $t$-linked overring of $R$ if $I^{-1}=R\Rightarrow IT$ invertible in $T$, for each finitely generated ideal $I$ of $R$ \cite{AHZ,KP}.

%%%%%%%%%%%%%%%%%%%%%%%%%%%%%%%%%%%%%%%%%%%%%%%%%%%%%%%%%%%%%%%%%%
%%%%%%%%%%%%%%%%%%%%%%%%%%%%%%%%%%%%%%%%%%%%%%%%%%%%%%%%%%%%%%%%%%
\begin{lemma}\label{sec:2.4}
Let $R$ be a GCD domain and $T$ a fractional overring of $R$ which is $t$-linked over $R$. Then $T$ is a GCD domain.
\end{lemma}

\begin{proof}
Since $R$ is a PVMD, by \cite[Proposition 2.10]{KP}, $T$ is $t$-flat over $R$, i.e., $R_{M}=T_{N}$ for each $t$-maximal ideal $N$ of $T$ and $M=N\cap R$. Moreover, since $T$ is $t$-linked over $R$, then $M_{t}\subsetneqq R$ \cite[Proposition 2.1]{DHLZ}. Hence $M$ is a
$t$-prime ideal of $R$ \cite[Corollary 2.47]{Kg}. Let $I$ be a finitely generated ideal of $T$. Then there exists a finitely generated ideal $J$ of $R$ such that $JT=I$. Since $R$ is a GCD domain, then $J_{t}=J_{v}=cR$, for some $c\in \qf(R)=\qf(T)$. Let $N\in\Max_{t}(T)$ and $M=N\cap R$. By \cite[Lemma 3.3]{KM2}, $IT_{N}=JR_{M}=J_{t}R_{M}=cR_{M}=cT_{N}$. We have $I_{v_{1}}=I_{t_{1}}=\bigcap_{N\in \Max_{t}(T)}IT_{N}=cT$ (which forces $c$ to lie in $T$). Therefore $T$ is a GCD domain.
\end{proof}

%%%%%%%%%%%%%%%%%%%%%%%%%%%%%%%%%%%%%%%%%%%%%%%%%%%%%%%%%%%%%%%%%%
%%%%%%%%%%%%%%%%%%%%%%%%%%%%%%%%%%%%%%%%%%%%%%%%%%%%%%%%%%%%%%%%%%
\begin{lemma}\label{sec:2.5}
Let $R$ be a domain and let $P\subseteq Q$ be two $t$-prime ideals of $R$ such that $R_{Q}$ is a valuation domain. Then $PR_{Q}=PR_{P}$.
\end{lemma}

\begin{proof}
We used this fact within the proof of \cite[Lemma 2.3]{KM3}. We reproduce here its proof for the sake of completeness. Clearly, $PR_{Q}\subseteq PR_{P}$. Assume $PR_{Q}\subsetneqq PR_{P}$ and $x\in PR_{P}\setminus PR_{Q}$. Then $PR_{Q}\subset xR_{Q}$
since $R_{Q}$ is a valuation domain. Hence, by \cite[Theorem 3.8 and Corollary 3.6]{HP}, $x^{-1}\in (R_{Q}:PR_{Q})=(PR_{Q}:PR_{Q})=(R_{Q})_{PR_{Q}}=R_{P}$, the desired contradiction.
\end{proof}

%%%%%%%%%%%%%%%%%%%%%%%%%%%%%%%%%%%%%%%%%%%%%%%%%%%%%%%%%%%%%%%%%%
%%%%%%%%%%%%%%%%%%%%%%%%%%%%%%%%%%%%%%%%%%%%%%%%%%%%%%%%%%%%%%%%%%
\begin{lemma}\label{sec:2.6}
 A domain $R$ is a strongly $t$-discrete PVMD if and only if $R_{P}$ is a strongly discrete valuation domain for every $t$-prime ideal $P$ of $R$.
\end{lemma}

\begin{proof}
Sufficiency is straightforward. Necessity. Let $P$ be a $t$-prime ideal of $R$. Assume there exists a $t$-prime ideal $Q\subseteq P$ such that $Q^{2}R_{P}=QR_{P}$. Let $M$ be an arbitrary $t$-maximal ideal of $R$ containing $Q$. By Lemma~\ref{sec:2.5}, we have $QR_{M}=QR_{Q}=QR_{P}=Q^{2}R_{P}=Q^{2}R_{Q}=Q^{2}R_{M}$. By \cite[Theorem 2.19]{Kg} or \cite[Theorem 6]{Ad}, $(Q^2)_{t}=Q$, absurd. So $R_{P}$ is a strongly discrete valuation domain.
\end{proof}

%%%%%%%%%%%%%%%%%%%%%%%%%%%%%%%%%%%%%%%%%%%%%%%%%%%%%%%%%%%%%%%%%%
%%%%%%%%%%%%%%%%%%%%%%%%%%%%%%%%%%%%%%%%%%%%%%%%%%%%%%%%%%%%%%%%%%
\begin{lemma}[{\cite[Lemma 2.8]{KM3}}]\label{sec:2.7}
Let $R$ be a PVMD and let $I$ be a $t$-ideal of $R$. Then:
\begin{enumerate}
\item $I$ is a $t$-ideal of $(I:I)$.
\item If $R$ is Clifford $t$-regular, then so is $(I:I)$.
\end{enumerate}
\end{lemma}

%%%%%%%%%%%%%%%%%%%%%%%%%%%%%%%%%%%%%%%%%%%%%%%%%%%%%%%%%%%%%%%%%%%%%%%%%%%%%%%%%%%%%%%
\begin{proof}[Proof of Theorem~\ref{sec:2.3}]
We need only prove the ``only if" assertion. Suppose $R$ is a Boole
$t$-regular and strongly $t$-discrete domain and let $I$ be a $t$-ideal of $R$. By
Theorem~\ref{sec:2.1}, $R$ is a GCD domain (and hence a PVMD). Moreover, $T:=(I:I)$ is a fractional $t$-linked overring of R (Cf. \cite[p. 1445]{KM3}). Hence $T$ is a GCD domain by Lemma~\ref{sec:2.4}. By Lemma~\ref{sec:2.7}, $I$ is a $t$-ideal of $T$. Suppose by way of contradiction that $J:=(I(T:I))_{t_{1}}\subsetneqq T$.

\begin{claim} $J$ is a fractional $t$-ideal of $R$.\end{claim}

Indeed, clearly $J$ is a fractional ideal of $R$. Let $x\in J_{t}$.
Then there exists a finitely generated ideal $B$ of $R$ such that
$B\subseteq J$ and $x(R:B)\subseteq R$. Similar arguments as above
yield $x\in \bigcap_{N\in \Max_{t}(T)}JT_{N}=J_{t_{1}}=J$. Therefore
$J=J_{t}$.

\begin{claim} $(J^{2})_{t_{1}}=cJ$ for some $0\not
=c\in \qf(R)$.\end{claim}

Indeed, there exists $0\not =c\in \qf(R)$ such that $(J^{2})_{t}=cJ$
since $R$ is Boole $t$-regular. Then $(J^{2})_{t_{1}}\subseteq
(cJ)_{t_{1}}=cJ$. Conversely, let $x\in cJ=(J^{2})_{t}$. Then there
exists a finitely generated ideal $A$ of $R$ such that $A\subseteq
J^{2}$ and $x(R:A)\subseteq R$. Similarly as above we get $x\in
\bigcap_{N\in \Max_{t}(T)}J^{2}T_{N}=(J^{2})_{t_{1}}$. Therefore
$(J^{2})_{t_{1}}=cJ$.

\begin{claim} $J$ is a $t$-idempotent $t$-ideal of $T$. \end{claim}

Indeed, since $J$ is a trace $t$-ideal of $T$ and $R$ is a Clifford
$t$-regular domain, we obtain
$(J:J)=(T:J)=(T:(I(T:I))_{t_{1}})=(T:I(T:I))=((I:I):I(T:I))=(I:I^{2}(I:I^{2}))=
(I:(I^{2}(I:I^{2})_{t})=(I:I)=T$. So
$(J:(J^{2})_{t_{1}})=(J:J^{2})=((J:J):J)=(T:J)=T$. Also
$(J:(J^{2})_{t_{1}})=(J:cJ)=c^{-1}(J:J)=c^{-1}T$. Therefore
$T=c^{-1}T$ and thus $c$ is a unit of $T$. Hence
$(J^{2})_{t_{1}}=J$, as claimed.

Now $J$ is a proper $t$-ideal of $T$, then $J$ is contained in a $t$-maximal ideal $N$ of $T$. Then $M=N\cap R$ is a $t$-prime ideal
of $R$ with $T_{N}=R_{M}$. By Lemma~\ref{sec:2.6}, $R_{M}=T_{N}$ is a strongly discrete valuation domain. However, Claim 3 combined with \cite[Lemma 3.3]{KM2} yields $J^{2}T_{N}=(J^{2})_{t_{1}}T_{N}=JT_{N}$. So $JT_{N}$ is an idempotent prime ideal of $T_{N}$ (since a valuation domain), the desired contradiction.

Consequently, $J=T$, i.e., $I$ is a $t$-invertible $t$-ideal of $T$. So there exists a finitely generated ideal $A$ of $T$ such that
$I=A_{v_{1}}$. Then there exists $a\in A$ such that $A_{v_{1}}=aT$
since $T$ is a GCD domain. Hence $I=aT$ and therefore $I$ is
strongly $t$-stable, completing the proof of the theorem.
\end{proof}

The next result shows that all the three notions, involved in Theorems \ref{sec:2.1} and \ref{sec:2.3}, collapse in the context of strongly $t$-discrete domains.

%%%%%%%%%%%%%%%%%%%%%%%%%%%%%%%%%%%%%%%%%%%%%%%%%%%%%%%%%%%%%%%%%%
%%%%%%%%%%%%%%%%%%%%%%%%%%%%%%%%%%%%%%%%%%%%%%%%%%%%%%%%%%%%%%%%%%
\begin{thm}\label{sec:2.8}
Let $R$ be a $v$-domain. The following assertions are equivalent:
\begin{enumerate}
\item $R$ is a strongly $t$-discrete Boole $t$-regular domain;
\item $R$ is a strongly $t$-discrete GCD domain of finite $t$-character;
\item $R$ is a strongly $t$-stable domain.
\end{enumerate}
\end{thm}

\begin{proof}
In view of Theorems \ref{sec:2.1} and \ref{sec:2.3}, we need only prove the implication (2)~$\Longrightarrow$~(3).
Let $I$ be a $t$-ideal of $R$. Then $I$ is a $t$-ideal of $T$ by Lemma~\ref{sec:2.7}. Set $T:=(I:I)$ and $J:=I(T:I)$.

\begin{claim} $T$ is strongly $t$-discrete. \end{claim}

Indeed, let $Q$ be a $t$-prime ideal of $T$. Then $P=Q\cap R$ is a
$t$-prime ideal of $R$ with $R_{P}=T_{Q}$ (see the proof of
Lemma~\ref{sec:2.4}). Assume by way of contradiction that
$(Q^2)_{t}=Q$. Then
$P^{2}R_{P}=Q^{2}T_{Q}=(Q^{2})_{t}T_{Q}=QT_{Q}=PR_{P}$ by
\cite[Lemma 3.3]{KM2}. Absurd since $R_{P}$ is strongly discrete by Lemma~\ref{sec:2.6}.

\begin{claim} $J_{t_{1}}=T$. \end{claim}

Indeed, since $R$ is a GCD of finite $t$-character, $R$ is
Clifford $t$-regular by \cite[Theorem 3.2]{KM2}. So, $I=(IJ)_{t}$.
Since $J$ is a trace ideal of $T$, then $T\subseteq
(T:J)=(J:J)\subseteq (IJ:IJ)\subseteq ((IJ)_{t}:(IJ)_{t})=(I:I)=T$,
hence $(T:J)=T$. Assume $J_{t_{1}}\subsetneqq T$. Then
$J\subseteq N$ for some $t$-maximal ideal $N$ of $T$. Hence
$T\subseteq (T:N)\subseteq (T:J)=T$ and so $(T:N)=(N:N)=T$. Then
$(N^{2}(N:N^{2}))_{t_{1}}=(N^{2})_{t_{1}}$. By Lemma~\ref{sec:2.7},
$T$ is Clifford $t$-regular. Therefore
$N=(N^{2}(N:N^{2}))_{t_{1}}=(N^{2})_{t_{1}}$, absurd since $T$ is
strongly $t$-discrete by Claim 4. Consequently,
$J_{t_{1}}=T$, proving the claim

Now $I$ is a $t$-invertible $t$-ideal of $T$ by Claim 5. So
$I=A_{t_{1}}=A_{v_{1}}$ for some finitely generated ideal $A$ of
$T$. Since $T$ is a GCD domain (Lemma~\ref{sec:2.4}), then
$I=A_{v_{1}}=cT$ for some $c\in T$, as desired.
\end{proof}

Next, as an application of the above theorem, we recover and improve \cite[Theorem 3.2]{KM,KM1} which provides a Boolean analogue for Olberding's result \cite[Theorem 4.6]{O1} on Pr\"ufer domains. The corollary also may be viewed as an analogue for Bazzoni's result \cite[Theorem 4.5]{Ba4} in the context of strongly discrete domains.

%%%%%%%%%%%%%%%%%%%%%%%%%%%%%%%%%%%%%%%%%%%%%%%%%%%%%%%%%%%%%%%%%%
%%%%%%%%%%%%%%%%%%%%%%%%%%%%%%%%%%%%%%%%%%%%%%%%%%%%%%%%%%%%%%%%%%
\begin{corollary}\label{sec:2.9}
Let $R$ be an integrally closed domain. The following assertions are equivalent:
\begin{enumerate}
\item $R$ is a strongly discrete Boole regular domain;
\item $R$ is a strongly discrete B\'ezout domain of finite character;
\item $R$ is a strongly stable domain.
\end{enumerate}
\end{corollary}

We close this paper with a simple method to build a new family of Boole $t$-regular domains originating from the class of GCD domains via Theorem~\ref{sec:2.8}.

%%%%%%%%%%%%%%%%%%%%%%%%%%%%%%%%%%%%%%%%%%%%%%%%%%%%%%%%%%%%%%%%%%
%%%%%%%%%%%%%%%%%%%%%%%%%%%%%%%%%%%%%%%%%%%%%%%%%%%%%%%%%%%%%%%%%%
\begin{example}\label{sec:2.10}
Let $V$ be a strongly discrete valuation domain with dimension $\geq2$, $n$ an integer $\geq2$, and $X_{1}, \cdots, X_{n-1}$ indeterminates over $V$. Then $R:=V[X_{1}, \cdots, X_{n-1}]$ is an $n$-dimensional Boole $t$-regular domain.
\end{example}

To prove this, we first  establish the following lemma, which is a re-phrasing of Statement (3) in \cite[Lemma 3.1]{KM3} and where we substitute the assumption ``integrally closed domain" to ``valuation domain."

%%%%%%%%%%%%%%%%%%%%%%%%%%%%%%%%%%%%%%%%%%%%%%%%%%%%%%%%%%%%%%%%%%
%%%%%%%%%%%%%%%%%%%%%%%%%%%%%%%%%%%%%%%%%%%%%%%%%%%%%%%%%%%%%%%%%%
\begin{lemma}\label{sec:2.11} Let $R$ be an integrally closed domain and $X$ an indeterminate over $R$. Then $R$ is strongly $t$-discrete if and only if $R[X]$ is strongly $t$-discrete.
\end{lemma}

\begin{proof}
Necessity. Assume $R$ is strongly $t$-discrete and $P$  is a $t$-idempotent $t$-prime ideal of $R[X]$ with $p:=P\cap R$.
If $p=(0)$ and $S:=R\setminus\{0\}$, then by \cite[Lemma 2.6]{KM2} $S^{-1}P$ is an
idempotent (nonzero) ideal of $S^{-1}R=\qf(R)[X]$, absurd. If $p\not=(0)$, then $P=p[X]$ with $p$ a $t$-ideal of $R$ (Cf. \cite{Q}). Hence $p[X]=((p[X])^{2})_{t}=(p^{2}[X])_{t}=(p^{2})_{t}[X]$, whence $p=(p^{2})_{t}$, absurd as desired. Sufficiency is straightforward.
\end{proof}

%%%%%%%%%%%%%%%%%%%%%%%%%%%%%%%%%%%%%%%%%%%%%%%%%%%%%%%%%%%%%%%%%%%%%%%%%%%%%%%%%%%%%%%
\begin{proof}[Proof of Example~\ref{sec:2.10}]
Clearly, $R$ is an $n$-dimensional GCD domain (which is not B\'ezout). Moreover, $R$ has finite $t$-character by \cite[Proposition 4.2]{KM2}. Finally, the strongly $t$-discrete condition is ensured by Lemma~\ref{sec:2.11}.
\end{proof}

%%%%%%%%%%%%%%%%%%%%%%%%%%%%%%%%%%%%%%%%%%%%%%%%%%%%%%%%%%%%%%%%%%%%%%%%%%%%%%%%%%%%%%%%%%%%%%%%%%%%%%%%%%%%%%%%%%%%%%%%%%%%%%%%%%%%%%%%%%%%%%
%%%%%%%%%%%%%%%%%%%%%%%%%%%%%%%%%%%%%%%%%%%%%%%%%%%%%%%%%%%%%%%%%%%%%%%%%%%%%%%%%%%%%%%%%%%%%%%%%%%%%%%%%%%%%%%%%%%%%%%%%%%%%%%%%%%%%%%%%%%%%%
%%%%%%%%%%%%%%%%%%%%%%%%%%%%%%%%%%%%%%%%%%%%%%%%%%%%%%%%%%%%%%%%%%%%%%%%%%%%%%%%%%%%%%%%%%%%%%%%%%%%%%%%%%%%%%%%%%%%%%%%%%%%%%%%%%%%%%%%%%%%%%
%%%%%%%%%%%%%%%%%%%%%%%%%%%%%%%%%%%%%%%%%%%%%%%%%%%%%%%%%%%%%%%%%%%%%%%%%%%%%%%%%%%%%%%%%%%%%%%%%%%%%%%%%%%%%%%%%%%%%%%%%%%%%%%%%%%%%%%%%%%%%%

%%%%%%%%%%%%%%%%%%%%%%%%%%%%%%%%%%%%%%%%%%%%%%%%%%%%%%%%%%%%%%%%%%%%%%%%%%%%%%%%%%%%%%%%%%%%%%%%%%%%%%%%%%%%%%%%%%%%%%%%%%%%%%%%%%%%%%%%%%%%%%
%%%%%%%%%%%%%%%%%%%%%%%%%%%%%%%%%%%%%%%%%%%%%%%%%%%%%%%%%%%%%%%%%%%%%%%%%%%%%%%%%%%%%%%%%%%%%%%%%%%%%%%%%%%%%%%%%%%%%%%%%%%%%%%%%%%%%%%%%%%%%%
%%%%%%%%%%%%%%%%%%%%%%%%%%%%%%%%%%%%%%%%%%%%%%%%%%%%%%%%%%%%%%%%%%%%%%%%%%%%%%%%%%%%%%%%%%%%%%%%%%%%%%%%%%%%%%%%%%%%%%%%%%%%%%%%%%%%%%%%%%%%%%
%%%%%%%%%%%%%%%%%%%%%%%%%%%%%%%%%%%%%%%%%%%%%%%%%%%%%%%%%%%%%%%%%%%%%%%%%%%%%%%%%%%%%%%%%%%%%%%%%%%%%%%%%%%%%%%%%%%%%%%%%%%%%%%%%%%%%%%%%%%%%%
\end{document}